\documentclass{article}
\usepackage{amsmath,amsthm,amssymb,xcolor}

\usepackage{subfigure,comment}

\usepackage{graphicx}
\graphicspath{
 {./}
}

\usepackage{url}
\usepackage{wrapfig}

\newtheorem{theorem}{Theorem}

\theoremstyle{definition}
\newtheorem{definition}[theorem]{Definition}

\theoremstyle{remark}

\usepackage
  [breaklinks,bookmarks,bookmarksnumbered,bookmarksopen,bookmarksopenlevel=2]
  {hyperref}
{\makeatletter \hypersetup{pdftitle={\@title}}}

{\makeatletter
 \gdef\xxxmark{%
   \expandafter\ifx\csname @mpargs\endcsname\relax 
     \expandafter\ifx\csname @captype\endcsname\relax 
       \marginpar{xxx}
     \else
       xxx 
     \fi
   \else
     xxx 
   \fi}
 \gdef\xxx{\@ifnextchar[\xxx@lab\xxx@nolab}
 \long\gdef\xxx@lab[#1]#2{\textbf{[\xxxmark #2 ---{\sc #1}]}}
 \long\gdef\xxx@nolab#1{\textbf{[\xxxmark #1]}}

\def\vec{\mathbf}

\newcommand{\trans}{^\textup{\textsf{T}}}

\title{A Proper Definition of Higher Order Rigidity}
\author{Tomohiro Tachi}

\begin{document}
\maketitle
\begin{abstract}
\cite{connelly-servatius:1994} shows the difficulty of properly defining $n$-th order rigidity and flexiblity of a bar-and-joint framework for higher order ($n\geq3$) through the introduction of a cusp mechanism.
I propose a ``proper'' definition of the order of rigidity by the order of elongation of the bars with respect to the arclength along the path in the configuration space.
We show that the classic definition using formal $n$-th derivative of the length constraint is a sufficient condition for the $n$-th flexiblity in the proposed definition and also a necessary condition only for $n=1,2$.
\end{abstract}
\section{Introduction}
Classicly, $n$-th order infinitesimal flexiblity [or rigidity] is defined by the existence [or nonexistence] of nominal $k$-th derivative of the vertex coordniates for $k=1,\dots n$, called $k$-th order flex, that make the $1$-st to $n$-th order derivative of length constraints being $0$, and that the first order flex is nontrivial.
This is a natural definition derived by the Taylor series expansion of the edge length.
Intuitively, the higher the order, the closer the system is to a finite mechanism.
So, the proper definition should satsfy the following: (1) $n$-th order rigidity implies $n+1$-th order rigidity and (2) $n$-th order rigidity implies finite rigidity (equivalently, (1') $n+1$-th order flexiblity implies $n$-th orider rigidity and (2) finite flexiblity implies $n$-th order flexiblity).
However, the cusp mechanism introduced in \cite{connelly-servatius:1994} is a counterexample, which is third order rigid but is finitely flexible; so the classic definition of higher order rigidity is not proper.

I propose a ``proper'' definition of $n$-th order rigidity using the arclength along the configuration space to solve this issue.
The proposed definition is compatible with the classic definition, in the sense that the classic definition is a sufficient condition for the $n$-th flexiblity in the proposed definition, and that the first and second order rigidity are equivalent in both definitions.

\section{Definition}
\begin{definition}[Finite Flexibility of a bar-and-joint framework]
A \emph{bar-and-joint framework} or a \emph{truss}, is a system of $V$ vertices in $\mathbb R^3$ and $E$ \emph{bars}.
Each vertex coordinate is $\vec x_u$ ($u=1, \dots, V$), and a bar $e_{u,v}$ connecting betweenvertices $u$ and $v$ have prescribed length $\ell_{u,v}$.
A \emph{configuration} is a set of $\vec x_u$  ($u=1, \dots, V$) satisfying the length of $e_{u,v}$ being $\ell_{u,v}$.
The configuration can be described as a point $\vec X$ in $\mathbb R^{3V}$ by concatenating the vertex coordinates as a single column.
The \emph{configuration space} of the truss is the set of configurations in $\mathbb R^{3V}$.
Here, we assume that sufficient number of vertices are fixed in space, so that the rigid body motion of the whole cannot happen.
We say the system \emph{(finitely) flexible} at a configuration $\vec X_0$ if there is a non-trivial $C^\infty$ continuous path $\vec X(t)$ in the configuration space passing through $\vec X(0) = \vec X_0$ that lies entirely in the configuration space; otherwise it is \emph{(finitely) rigid}.
\end{definition}
Precisely, let $\vec x_u = [x_u, y_u, z_u]\trans$ denote the coordinate of vertex $u$ ($u=1, \dots, V$);
then the configuration of the truss is
\begin{align*}
\vec X &= 
\begin{bmatrix}
\vec x_1 \\ \vdots \\ \vec x_u \\ \vdots \\ \vec x_V
\end{bmatrix}
=
\begin{bmatrix}
x_1\\ y_1 \\ z_1 \\ \vdots \\ x_V \\ y_V \\ z_V
\end{bmatrix},
\end{align*}
Also, let the $e$-th edge of all $E$ edges be incident to $u$ and $v$ and has prescribed length $\ell_{u,v}$.
Then, the length of the edge is preserved and thus the elongation measured by the squred length is
$$
D_e (\vec X) := \left(\vec x_u - \vec x_v \right)\trans\left(\vec x_u - \vec x_v \right) - \ell^2_{u,v},
$$
must be $0$ for each edge.
Note that $D_e$ is a proper measure for the actual elongation $d_e:=\left\|\vec x_u - \vec x_v \right\| - \ell_{u,v}$ because
$D_e=(\left\|\vec x_u + \vec x_v \right\| + \ell_{u,v}) d_e = \Theta(d_e)$ as $t$ approaches to $0$ because the bar lengths are finite and $\left\|\vec x_u + \vec x_v \right\| + \ell_{u,v}\neq 0$.
This can be respresented as a vector equation $\vec D(\vec X) = \vec 0_E$, 
where $e$-th element of $\vec D$ is $D_e$, i.e.,
\begin{align*}
\vec D(\vec X) &= 
\begin{bmatrix}
D_1(\vec X) \\ \vdots \\ D_e(\vec X) \\ \vdots \\ D_E(\vec X).
\end{bmatrix}
\end{align*}
\begin{definition}[Proposed Definition of $n$-th order rigidity]
A pin-and-joint framework is \emph{$n$-th order infinitesimally flexible} ($n\in \mathbb N$) at a configuration $\vec X_0$ if there is a non-trivial $C^\infty$ continuous path $\vec X(t)$ in the configuration space passing through $\vec X(0) = \vec X_0$ that keeps the order of the elongation $D_e$ for every edge $e$ in $o(s^n)$ as $s$ approaches to $0$, where $s=s(t)=\int_0^t \frac{d\left\vert\vec X(t)\right\vert}{dt} dt$ is the arclength along the path%
\footnote{Here, we discuss the order of error as $t$ approaches to $0$. The big O notation $f(t)=O(g(t))$ represents the upper bound, i.e., there exists a constant $k>0$ such that $|g(t)| \leq k\left|f(t)\right|$ for a sufficiently small $t>0$, while small o notation $f(t)=o(g(t))$ represents that $g$ dominates $f$, i.e., for any value $k>0$, $|g(t)| \leq k\left|f(t)\right|$ for a sufficiently small $t>0$. For example, $t^2 = O(t^2) = o(t)$. 
Also, big Omega $f(t)=\Omega(g(t))$ is used to represent the lower bound i.e., $k>0$ such that $|g(t)| \geq k\left|f(t)\right|$ for a sufficiently small $t>0$. If $f(t)=O(g(t))$ and $f(t)=\Omega(g(t))$ at the same time, we get a tight order expression $f(t) = \Theta(g(t))$.}%
, otherwise it is \emph{$n$-th order infinitesimally rigid}.
\end{definition}
\begin{theorem}[Proper]
By definition, $n$-th order rigidity implies $n+1$-th order rigidity.
$n$-th order rigidity implies finite rigidity.
\end{theorem}

\section{Flex and Flexibility}
Now we compare the proposed definition with the classic definition.
Consider a motion $\vec X(t)$ where $t=0$ represents the initial state.
The the change of squared length of edge $e=(u,v)$ is $D_e(t) = D_e(\vec X(t))$; then $D_e(0) = 0$.
Now the Taylor series expansion of the squred length gives
\begin{align}
\begin{split}
\label{eq:expansion}
D_e (t) &= \sum_{k=1}^{n} \frac{1}{k!} \sum_{a=0}^{k} \binom{k}{a} \left(\vec x_u^{(a)} - \vec x_v^{(a)} \right)\trans\left(\vec x_u^{(k-a)} - \vec x_v^{(k-a)} \right)t^k + o(t^n)
\end{split}
\end{align}

\begin{definition}[Classic Definition]
The formal derivatives $\left(\vec X^{(1)} \dots, \vec X^{(n)}\right)$ satisfing
\begin{align}
\sum_{a=0}^{k} \binom{k}{a} \left(\vec x_u^{(a)} - \vec x_v^{(a)} \right)\trans\left(\vec x_u^{(k-a)} - \vec x_v^{(k-a)} \right)&=0,
\end{align}
where $\vec X^{(0)} := \vec X$, for every bar connecting $u$ and $v$ and for all $k=1, \dots, n$ is called \emph{$n$-th order flex}.
\end{definition}

\begin{theorem}[Classic $n$-th Order Flexiblity is Sufficient]
A bar-and-joint framework is $n$-th order flexible if there is $n$-th order flex $\left(\vec X^{(1)} \dots, \vec X^{(n)}\right)$ with nontrivial first order flex $\vec X^{(0)}\neq \vec 0$.
\end{theorem}
\begin{proof}
Given $n$-th order flex, $\left(\vec X^{(1)} \dots, \vec X^{(n)}\right)$, construct a path $\vec X(t) = \vec X_0 + \vec X^{(1)}t + \frac{1}{2}\vec X^{(2)}t^2 + \dots + \frac{1}{n!} \vec X^{(n)}t^n$, so that $k$-th derivative of the curve at the initial state equals the $k$-th term of the given flex.
Then, $D_e(t) - \ell^2_e = o(t^n)$.
The arclength along this path is first order with respect to $t$, i.e., $s = \Theta(t)$ because $\frac{ds}{dt} = \left\|\vec X^{(1)}\right\|\neq 0$.
Therefore, $D_e(t) - \ell^2_e = o(s^n)$, and thus it is $n$-th order flexible.
\end{proof}

\begin{theorem}[First and Second Order]
A bar-and-joint framework is first order flexible if and only if there is nontrivial first order flex $\vec X^{(1)}$.
A bar-and-joint framework is second order flexible if there is second order flex $\left(\vec X^{(1)}, \vec X^{(2)}\right)$ with nontrivial first order flex $\vec X^{(1)}$.
\end{theorem}
\begin{proof}
For necessity for $n=1,2$ cases, consider that $s=\Theta(t^n)$ and thus $\frac{d^{k}s}{dt^{k}}=0$ ($k=1,\dots,n-1$) and $\frac{d^{n}s}{dt^{n}}\neq 0$.
Then, take the first $2n$ term of~(\ref{eq:expansion}).
\begin{align}
\begin{split}
\label{eq:n-term}
D_e (t)
&= \frac{2}{n!} \left(\vec x_u^{(0)} - \vec x_v^{(0)} \right)\trans\left(\vec x_u^{(n)} - \vec x_v^{(n)} \right)t^{n}
+ o(t^n)\\ 
\end{split}
\intertext{and}
\begin{split}
\label{eq:2n-term}
D_e (t) 
&= \sum_{k=0}^{n-1} \frac{2}{(n+k)!} \left(\vec x_u^{(0)} - \vec x_v^{(0)} \right)\trans\left(\vec x_u^{(n+k)} - \vec x_v^{(n+k)} \right)t^{n+k}\\ 
&+ \frac{1}{2n!}\left(
2\left(\vec x_u^{(0)} - \vec x_v^{(0)} \right)\trans \left(\vec x_u^{(2n)} - \vec x_v^{(2n)} \right) 
+
\binom{2n}{n}\left(\vec x_u^{(n)} - \vec x_v^{(n)} \right)\trans \left(\vec x_u^{(n)} - \vec x_v^{(n)} \right) 
\right) t^{2n}\\
&+ o(t^{2n})
\end{split}
\end{align}
If we substitute $\rho_i^{(n)} \rightarrow \rho_i^{(1)}$ and $\frac{1}{2}\binom{2n}{n} \rho_i^{(2n)} \rightarrow \rho_i^{(2)}$, then $t^n$ term of~(\ref{eq:n-term}) eliminates only if there exists a non-trivial first order flex, and $t^{2n}$ term of~(\ref{eq:2n-term}) eliminates only if there exists a second order flex with non-trivial first order flex.
\end{proof}

Notice that the approach used in the proof does not extend any longer.
For example, the coefficient of $t^{3n}$ term eliminates only if
\begin{align}
\begin{split}
& 2 \left(\vec x_u^{(0)} - \vec x_v^{(0)} \right)\trans\left(\vec x_u^{(3n)} - \vec x_v^{(3n)} \right)
+2\binom{3n}{n} \left(\vec x_u^{(n)} - \vec x_v^{(n)} \right)\trans\left(\vec x_u^{(2n)} - \vec x_v^{(2n)} \right)\\
&+ \sum_{k=1}^{n-1} \binom{3n}{n+k} \left(\vec x_u^{(n+k)} - \vec x_v^{(n+k)} \right)\trans\left(\vec x_u^{(2n-k)} - \vec x_v^{(2n-k)} \right)
=0.
 \end{split}
\end{align}
The first two terms represent the third order flex, but there are extra terms making the equation satisfy in a different way making use of from $n+1$ to $2n-1$ terms.
This is an interpretation of the result of~\cite{connelly-servatius:1994} that third order flexiblity cannot be defined by the existence of third order flex. 

\section{Cusp Mechanism}
Let us check how the ``flexes'' of a cusp mechanism could be represented in the example of connected Watt's mechanism introduced in~\cite{connelly-servatius:1994}.
Refer to Figure~\ref{fig:cusp-mechanism}.
\begin{figure}[htbp]
\begin{center}
  \includegraphics[width=1.0\linewidth,page=1]{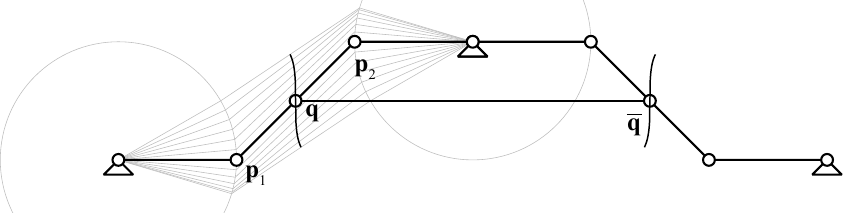}
\end{center}
  \caption{Cusp mechanism constructed by connecting a pair of Watt's mechanisms.}
 \label{fig:cusp-mechanism}
\end{figure}

In this mechanism, we assume the parameterization $s = \Theta(t^2)$.
Then, the derivative of $D_e$ up to $6$-th order is
\begin{align}
\label{eq:d1}
 D_e^{(1)}(0) &=  0\\
\label{eq:d2}
\frac {1}{2} D_e^{(2)}(0) &=  \left(\vec x_u^{(0)} - \vec x_v^{(0)} \right)\trans\left(\vec x_u^{(2)} - \vec x_v^{(2)} \right)\\
\label{eq:d3}
\frac {1}{2} D_e^{(3)}(0) &=  \left(\vec x_u^{(0)} - \vec x_v^{(0)} \right)\trans\left(\vec x_u^{(3)} - \vec x_v^{(3)} \right)\\
\begin{split}
\label{eq:d4}
\frac {1}{2} D_e^{(4)}(0) 
&=  \left(\vec x_u^{(0)} - \vec x_v^{(0)} \right)\trans\left(\vec x_u^{(4)} - \vec x_v^{(4)} \right) \\
&+3 \left(\vec x_u^{(2)} - \vec x_v^{(2)} \right)\trans\left(\vec x_u^{(2)} - \vec x_v^{(2)} \right)
\end{split}\\
\begin{split}
\label{eq:d5}
\frac {1}{2} D_e^{(5)}(0) 
&=  \left(\vec x_u^{(0)} - \vec x_v^{(0)} \right)\trans\left(\vec x_u^{(5)} - \vec x_v^{(5)} \right) \\
&+10 \left(\vec x_u^{(2)} - \vec x_v^{(2)} \right)\trans\left(\vec x_u^{(3)} - \vec x_v^{(3)} \right)
\end{split}\\
\begin{split}
\label{eq:d6}
\frac {1}{2} D_e^{(6)}(0) 
&=  \left(\vec x_u^{(0)} - \vec x_v^{(0)} \right)\trans\left(\vec x_u^{(6)} - \vec x_v^{(6)} \right) \\
&+15 \left(\vec x_u^{(2)} - \vec x_v^{(2)} \right)\trans\left(\vec x_u^{(4)} - \vec x_v^{(4)} \right)\\
&+10 \left(\vec x_u^{(3)} - \vec x_v^{(3)} \right)\trans\left(\vec x_u^{(3)} - \vec x_v^{(3)} \right)
\end{split}
\end{align}
The mechanism is third-order flexible if and only if~(\ref{eq:d1})--(\ref{eq:d6}) are zero.

Point $\vec p_1$ rotating along a circle, whose radius bar direction is initially $\begin{bmatrix}1\\0\end{bmatrix}$,
should have the following flexes:
\begin{align}
\begin{split}
&\left(
\vec p_1^{(1)},
\vec p_1^{(2)},
\vec p_1^{(3)},
\vec p_1^{(4)},
\vec p_1^{(5)},
\vec p_1^{(6)}
\right)
=\\
&\left(
\begin{bmatrix}0\\0\end{bmatrix},
\begin{bmatrix}0\\a_1\end{bmatrix},
\begin{bmatrix}0\\b_1\end{bmatrix},
\begin{bmatrix}-3a_1^2\\c_1\end{bmatrix},
\begin{bmatrix}-10a_1b_1^2\\d_1\end{bmatrix},
\begin{bmatrix}-15a_1c_1-10b_1^2\\e_1\end{bmatrix}
\right)
\end{split}
\end{align}
By symmetry, point $\vec p_2$ has the following $6$-th order flexes:
\begin{align}
\begin{split}
&\left(
\vec p_2^{(1)},
\vec p_2^{(2)},
\vec p_2^{(3)},
\vec p_2^{(4)},
\vec p_2^{(5)},
\vec p_2^{(6)}
\right)
=\\
&\left(
\begin{bmatrix}0\\0\end{bmatrix},
\begin{bmatrix}0\\a_2\end{bmatrix},
\begin{bmatrix}0\\b_2\end{bmatrix},
\begin{bmatrix}3a_2^2\\c_2\end{bmatrix},
\begin{bmatrix}10a_2 b_2\\d_2\end{bmatrix},
\begin{bmatrix}15a_2 c_2+10b_2^2\\e_2\end{bmatrix}
\right)
\end{split}
\end{align}
Now, consider the bar connecting $\vec p_1$ and $\vec p_2$, whose initial direction is $(\vec p_2-\vec p_1)=\begin{bmatrix}1\\1\end{bmatrix}$.
By~(\ref{eq:d2}) and~(\ref{eq:d3}) being $0$, we get $a_1= a_2 = a$ and $b_1=b_2=b$.
By~(\ref{eq:d4}) and~(\ref{eq:d5}) being $0$, we get $c_1-c_2 = 6a^2$ and $d_1-d_2=20ab$.
By~(\ref{eq:d6}) being $0$, we get $e_1-e_2 = 15a(c_1+c_2) + 20b^2$.

The midpint $\vec q$ of $\vec p_1$ and $\vec p_2$ have the following flexes:
\begin{align}
\begin{split}
&\left(
\vec q^{(1)},
\vec q^{(2)},
\vec q^{(3)},
\vec q^{(4)},
\vec q^{(5)},
\vec q^{(6)}
\right)
=\\
&\left(
\begin{bmatrix}0\\0\end{bmatrix},
\begin{bmatrix}0\\a\end{bmatrix},
\begin{bmatrix}0\\b\end{bmatrix},
\begin{bmatrix}0\\ \frac{1}{2}\left(c_1 + c_2\right)\end{bmatrix},
\begin{bmatrix}0\\ \frac{1}{2}\left(d_1 + d_2\right) \end{bmatrix},
\begin{bmatrix}-45a^3\\ \frac{1}{2}\left(e_1 + e_2\right)\end{bmatrix}
\right)
\end{split}
\end{align}

By symmetry, the midpint $\overline{\vec q}$ of the other Watt's mechanism have the following flexes:
\begin{align}
\begin{split}
&\left(
\overline{\vec q}^{(1)},
\overline{\vec q}^{(2)},
\overline{\vec q}^{(3)},
\overline{\vec q}^{(4)},
\overline{\vec q}^{(5)},
\overline{\vec q}^{(6)}
\right)
=\\
&\left(
\begin{bmatrix}0\\0\end{bmatrix},
\begin{bmatrix}0\\ \overline{a}\end{bmatrix},
\begin{bmatrix}0\\ \overline{b}\end{bmatrix},
\begin{bmatrix}0\\ \frac{1}{2}\left(\overline{c_1} + \overline{c_2}\right)\end{bmatrix},
\begin{bmatrix}0\\ \frac{1}{2}\left(\overline{d_1} + \overline{d_2}\right) \end{bmatrix},
\begin{bmatrix}45\overline{a}^3\\ \frac{1}{2}\left(\overline{e_1} + \overline{e_2}\right)\end{bmatrix}
\right)
\end{split}
\end{align}
Horizontal bar connecting $\vec q$ and $\overline{\vec q}$ has the direction of $\begin{bmatrix}1\\0\end{bmatrix}$.
For this bar, (\ref{eq:d1})-(\ref{eq:d3}) are all $0$.
Now from~(\ref{eq:d4}) being $0$, we get $a=\overline{a}$, which automatically make~(\ref{eq:d5}) $0$.
The last term~(\ref{eq:d6}) being $0$ gives $9a^3 + (\overline{b}-b)^2 = 0$.
This equation well describes the assymetric behavior of the mechanism at the cusp point (Figure~\ref{fig:cusp-details});
\begin{enumerate}
\item it forces that $a < 0$, meaning that the horizontal bar must drop down instead of going up; and
\item the sign of $\overline{b}-b = \pm \sqrt{-a^3}$ determines which way the bar is going to be tilted.
\end{enumerate}
\begin{figure}[htbp]
\begin{center}
  \includegraphics[width=0.75\linewidth,page=2]{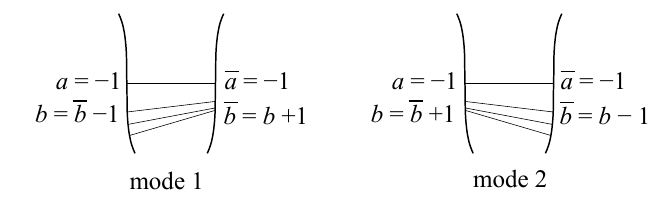}
\end{center}
  \caption{The bifurcation of the motion of horizontal bar.}
 \label{fig:cusp-details}
\end{figure}
\section{Discussion}
The definition of measuring the order of edge elongation with respect to the arclength along the kinematic path solves the issue of infinitesimal rigidity.
There are variants of the measure we can use instead of arclength, e.g.,  the displacement $d = \left\|\vec X(t) - \vec X_0\right\|$ because they are related to each other in first order.
In fact, the definition straightforwardly relates the displacement and the elongation of the element, so I believe it is a ``natural'' definition that engineers would agree.

Non integer number of degrees of rigidity can be potentially defined in the proposed definition.
An open question arises from here: are there systems with tightly non-integer order of rigidity, e.g., are there $3$-rd order flexible but $3.5$-th order rigid systems?

\end{document}